\documentclass[12pt]{amsart}
\usepackage{amsmath, amsfonts, amsthm, amssymb,txfonts}

\usepackage{graphicx}
\usepackage{color}

\usepackage{enumitem}

\numberwithin{equation}{section}
\numberwithin{figure}{section}
\theoremstyle{plain}
\newtheorem{thm}{\protect\theoremname}
  \theoremstyle{plain}
  \newtheorem{prop}[thm]{\protect\propositionname}
  \theoremstyle{plain}
  
  \theoremstyle{plain}
  \newtheorem{lem}[thm]{\protect\lemmaname}
  \theoremstyle{remark}
  \newtheorem{rem}[thm]{\protect\remarkname}
  \theoremstyle{definition}

  \providecommand{\corollaryname}{Corollary}
  \providecommand{\examplename}{Example}
  \providecommand{\lemmaname}{Lemma}
  \providecommand{\propositionname}{Proposition}
  \providecommand{\remarkname}{Remark}
  \providecommand{\theoremname}{Theorem}

  \newcommand{\E}{\mathcal{E}}
  \newcommand{\dom}{\operatorname{dom}}
  
  \newcommand{\DF}{\mathcal{E}}

\title[Maximum Principles for Schr\"odinger operators on fractals]{The
  Strong Maximum Principle for Schr\"{o}dinger operators on
  fractals} 
  \author{Marius~V. Ionescu}
  \address{Department of Mathematics,  United States Naval Academy, Annapolis,
  MD, 21402-5002, USA}
\email{ionescu@usna.edu}

\author{Kasso A.~Okoudjou}
\address{Department of Mathematics and Norbert Wiener Center\\
University of Maryland\\
College Park\\
MD 20742}
\email{kasso@math.umd.edu}

\author{ Luke~G. Rogers}
\address{Department of Mathematics, University of Connecticut, Storrs,
  CT 06269-1009, USA}
\email{rogers@math.uconn.edu}


\date{\today}


\begin{document}

\begin{abstract}
We prove a strong maximum principle for Schr\"odinger operators
defined on a class of fractal sets and their blowups without boundary.
Our primary interest is in weaker regularity conditions than have
previously appeared in the literature; in particular we permit both
the fractal Laplacian and the potential to be Radon measures on the
fractal. As a consequence of our results, we establish a Harnack
inequality for solutions of these operators.  
\end{abstract}

\maketitle

\section{Introduction}
\label{sec:introduction}
The goal of this note is to prove a strong maximum principle and related 
results for Schr\"odinger operators $L=\Delta -\nu$\/ where $\Delta$
is a fractal Laplacian (to be defined below) and the potential $\nu$
is a non-negative measure on the fractal.  When the fractal $K$ is the
standard Sierpinski gasket or its unbounded analogue, Strichartz
established strong maximum principles for solutions of the nonlinear
equation $\Delta u=F(x, u)$ where $F: K\times \mathbb{R} \to
\mathbb{R}$ is continuous and nonnegative in the sense that if
$u(x)\geq 0$ then $F(x, u(x))\geq 0$, see~\cite{Str:MP99}.  We consider
algebraically simpler operators because our interest is in weakening
the regularity conditions on both the Laplacian $\Delta$ and the
potential, both of which we will permit to be measures. 

In Section~\ref{sec:prelims} we recall some basic facts about analysis
on fractals and fractal blowups, details of which may be found
in~\cite{Kig_CUP01,Str_Prin06}.  Section~\ref{sec:mp} contains the
proof of the maximum principle and some comments on a Hopf-type lemma,
and Section~\ref{sec:harnack} has the proof of a Harnack inequality.

\section{Preliminaries}\label{sec:prelims}
We consider a self-similar set $K$ generated by an iterated function 
system $\{F_1,\dots,F_N\}$ consisting of contractive maps on a complete metric space.  
To a finite word
$\omega=\omega_1\omega_2\dotsm\omega_n\in\{1,\dotsc,N\}^n$ we
associate $F_\omega=F_{\omega_1}\circ\dotsm\circ F_{\omega_n}$ and a
cell $C_\omega=F_\omega(K)$.  We assume $K$ is post-critically finite,
whence there is a finite set $V_0$ so that for any words
$\omega,\omega'$ we have $C_\omega\cap C_{\omega'}\subset
F_\omega(V_0)\cap F_{\omega'}(V_0)$.  From the set $V_0$, we
inductively define $V_{n+1}=\cup_{i=1}^N F_i(V_n)$ and thereby obtain
a countable dense subset $V_*=\cup_n V_n$ of $K$.  The topology of $K$
is generated by cells in the sense that any $x\notin V_*$ has a
neighborhood base consisting of interiors of cells, while any $x\in
V_n$ has a neighborhood base in which each set is a finite union of
interiors of cells adjoined at $x$. We let $\mu$ denote the standard
self-similar measure on $K$. 

Following Kigami~\cite{Kig_CUP01} we make the strong assumption that
there is a resistance form $\E$, also called the energy, with domain
$\dom\E$  that is obtained from a regular self-similar harmonic
structure.  This means that there is an irreducible, non-negative,
symmetric, quadratic form $\E_0$ defined on the (finite-dimensional)
vector space of functions on $V_0$, and factors $0<r_i<1$ for each
$i=1,\dotsc, N$, such that setting $r_\omega=\prod_1^n r_{\omega_j}$
we have for $u\in\dom\E$ 
\begin{equation*}
	\E_n(u,u) := \sum_{\omega\in \{1,\dotsc,N\}^n} r_\omega^{-1} \E_0(u\circ F_\omega,u\circ F_\omega)
	\end{equation*}
and $\E(u,u)=\lim_n \E_n(u,u)$, where the latter sequence is
non-decreasing.  Those functions on which $\E_n$ is constant for
$n\geq m$ are called piecewise harmonic at scale $m$. It follows that
the pair $(\E,\dom\E)$ has the following properties: 1)  $\E$ is a
non-negative symmetric quadratic form on the linear space $\dom\E$, 2)
$\E$ vanishes exactly on the constants and $\dom\E$ modulo constants
is a Hilbert space under $\E$, 3) any function on a finite subset of
$K$ has an extension in $\dom\E$, 4) for any $p,q\in K$ the quantity  
\begin{equation}\label{eqn:defofR}
	R(p,q)=\sup\{|u(p)-u(q)|^2/\E(u,u): u\in\dom\E\text{ and } \E(u,u)>0\}
\end{equation}
 is finite, and 5) if $u\in\dom\E$ then so is
 $\bar{u}=\max\{0,\min\{u,1\}\}$ and $\E(\bar{u},\bar{u})\leq\E(u,u)$.
 Moreover, $R(p,q)$ is a metric on $V_*$ with completion homeomorphic
 to (and hence identified with) $K$, to which the continuous extension
 of $R$ is called the resistance metric; any function $u\in\dom\E$
 satisfies
 \begin{equation}\label{eqn:estres}
   |u(x)-u(y)|^2\leq \E(u,u)R(x,y).
 \end{equation}
It follows that any $u\in\dom\E$ is continuous with respect to the
 resistance metric; and $\E$ is a Dirchlet
 form on $L^2(K,\mu)$, see~\cite{Kig_93TAMS}.  From the general theory of Dirichlet forms (see, for example, \cite{FOT}) one
 then defines a non-positive definite, self-adjoint (Dirichlet)
 Laplacian operator $\Delta$ for which $u\in\dom\Delta$ and $\Delta
 u=f\in C(K)$ by requiring 
\[
\E(u,v)=-\int_K fv\,d\mu
\]
for all $v\in\dom{\E}$ such that $v|_{V_0}=0$. If we only assume that $f\in L^p(\mu)$, then we say that $u\in\dom_{L^p}\Delta$.  We will primarily work with an extension of the above definition in which
$\Delta u$ is a finite signed Radon measure:  we say that $u\in\dom_{\mathcal{M}}\Delta$ and $\Delta u=\sigma$ if
\[
  \E(u,v)=-\int_K v\,d\sigma
\]
for all $v\in\dom{\E}$ such that $v|_{V_0}=0$.   Of course, we can view any
function $f\in L^1(\mu)$ as the measure $f\,d\mu$. The Laplacian is
self-similar in the sense that
\[
  \Delta_{\mathcal{M}}(u\circ
  F_\omega)=r_\omega\mu_{\omega}(\Delta_{\mathcal{M}}u)\circ F_\omega,
  \]
  for all $\omega\in \{1,\dots,N\}^n$ and $n\ge 1$, where
   $\mu_i$ are the
  weights of the self-similar measure $\mu$ and $\mu_\omega=\prod_{j=1}^n \mu_{\omega_j}$ and $r_\omega$ is defined similarly.

Our main results are also valid on bounded subsets of the fractal blowups considered by Strichartz~\cite{Str_CJM98}.  An infinite blow-up without boundary points of the
p.c.f.\ fractal $K$ is defined using  a sequence $\alpha\in
\{1,\dots,N\}^\mathbb{N}$ such that $\alpha$ is \textbf{not}
eventually constant. For $n\ge 1$ set
$K_n=F_{\alpha_1}^{-1}F_{\alpha_2}^{-1}\dots
F_{\alpha_n}^{-1}(K)$. Then $\{K_n\}$ is an increasing sequence of
sets and the infinite blow-up is defined to be
$K_\infty=\bigcup_{n=1}^\infty K_n$. Both the energy
$\E$ and the measure $\mu$ extend to $K_\infty$ in the obvious fashion, and we write $\E_\infty$ and
$\mu_\infty$ for these extensions. The  Laplacian $\Delta_\infty$ is defined weakly as before. 

\section{Maximum principle}
\label{sec:mp}

Let  $\nu$ be a finite non-negative Radon measure.  The
Schr\"{o}dinger operator $L$ with potential $\nu$ is defined on $\dom_{\mathcal{M}}\Delta$ by
\begin{equation}\label{eqn:defL}
  Lu:= \Delta u- u \nu
\end{equation}
The main result of this section is
a maximum principle for the operator $L$. The proof goes via a result
on subharmonic functions that relies on an argument of
Kigami~\cite[Theorem~5.8(2)]{Kig_93TAMS}; it has been used in various
forms by other authors (for example, in
\cite[Theorem~2.1]{Str:MP99}, \cite[Lemma~4.3]{Rog_08}).  For the purposes of exposition we first prove a weak and then a strong maximum principle, though of course the former is a consequence of the latter.

\begin{prop}\label{lem:trivialmp}
Let  $u\in \dom_{\mathcal M}\Delta$.  If $C=C_\omega$ is a cell on which $\Delta u$ is a non-negative measure then $\max_C u\leq \max_{\partial C}u$. Moreover, if $u$ attains its maximum at an interior point of $C$ then $u$ is constant on $C$.
\end{prop}
\begin{proof}
Self-similarity of $\Delta$ ensures there is no loss of generality in taking $C=K$.
For any $n$ and a point $p\in V_n\setminus V_0$ there is a  piecewise
harmonic function at scale $n$ called $h_p$ that is harmonic on
$K\setminus V_n$ with $h_p(p)=1$ and $h_p(q)=0$ if $q\in V_n\setminus
\{p\}$.  The maximum principle for harmonic functions (\cite[Theorem 3.2.5]{Kig_CUP01}) implies
$0\leq h_p\leq 1$, and since $h_p|_{V_0}=0$ we have 
\begin{equation}\label{eqn:trivialmp}
  0=\E(h_p,u)+\int h_p\Delta u \geq
  \E(h_p,u)=\E_n(h_p,u)=\sum_{q\sim_n
    p}c_{pq}\bigl(u(p)-u(q)\bigr), 
\end{equation}
where the $c_{pq}>0$ are constants depending only on $\E_n$ and
$q\sim_n p$ means that $q$ and $p$ are neighbors in $V_n$.  Using
this inductively beginning at $n=1$ and considering each point in
$V_n$  we deduce that $u$ is bounded on $V_*$ by $\max_{V_0}u$, whence
the desired inequality follows by continuity of $u$. 

Now suppose $u$ attains an interior maximum at $x$.  We distinguish two cases according to whether $x\in V_*$ or $x\notin V_*$.

If $x\in V_*$ then it is in $V_n$ for some $n$. Let $q\sim_n x$ and
from~\eqref{eqn:trivialmp} and $u(x)\geq u(q)$ deduce both that
$u(q)=u(x)$ and that $\int h_x \Delta u=0$.  Since $h_x>0$ on the
interiors of the $n$-cells containing $x$ we conclude that $\Delta u$
has no mass on these cells and thus that $u$ is harmonic on them.  As
$u$ is harmonic and all its boundary values on these cells equal
$u(x)$ it is a constant function, and hence $u\equiv u(x)$ on a
neighborhood of $x$. 

If $x\notin V_*$ we fix an $n$ and the $n$-cell $C_n$ containing $x$.
Let $h$ be harmonic on $C_n$ with $h=u$ on the boundary $\partial
C_n$.  We use another result of Kigami~\cite[Proposition~3.5.5 and Theorem~3.5.7]{Kig_CUP01}, namely that
there is a non-negative Green kernel $g$ that inverts $-\Delta$ on
$C_n$ with Dirichlet boundary data. Thus, for $y\in C_n$, 
\begin{equation}\label{eqn:subharmoncell}
	u(y)=h(y)+\int g(y,z)(-\Delta u(dz))\leq h(y),
	\end{equation}
 which simply says $u$ is subharmonic.  However we then have
\begin{equation*}
	u(x)\leq h(x) \leq \max_{\partial C_n} h(y) =\max_{\partial C_n} u(y)\leq u(x)
	\end{equation*}
where the first inequality is~\eqref{eqn:subharmoncell}, the second is
the maximum principle for harmonic functions, the equality is $u=h$ on
$\partial C_n$, and the final inequality is that $u(x)$ is the maximum
of $u$.  Since equality must hold throughout we conclude $u=u(x)$ on
$\partial C_n$. However~\eqref{eqn:subharmoncell} must also be an
equality, and since $g(y,z)>0$ unless $z\in\partial C$ we find that
$\Delta u$ has no mass on the interior of $C$ whence $u$ is harmonic
and therefore constant.  Again we have found $u\equiv u(x)$ on a
neighborhood of $x$. 

We conclude by noting $\{y:u(y)=u(x)\}$ is closed by continuity of
$u$, open by the preceding reasoning, and contains $x$, so by
connectivity it is $K$. 
\end{proof}

The preceding argument extends readily to our class of  Schr\"odinger
operators. For $u$ a function on $K$ let  $u^{+}(x)=\max \{u(x),0\}$. 
  
\begin{thm}[Maximum Principle for Schr\"odinger operators]\label{thm:wmp}
Let $\nu$ be a non-negative Radon measure on $K_\infty$. Suppose $E\subset
K_\infty$ is open and bounded, and consider $u\in \dom_{\mathcal
  M}\Delta_\infty$.  If $Lu=\Delta u- \nu u$ is a non-negative measure 
on $E$  then $$\max_{x\in E}u(x)\le  \max_{p\in \partial E}u^{+}(p).$$
Moreover, if $u$ achieves a positive maximum at an interior point
$x\in E$ then $u$ is constant on the connected component of $E$
containing $x$. 
\end{thm}
\begin{proof}
Observe that the asserted inequality is trivial if $u\leq0$.
Accordingly we may assume  $U:=\{x\in E\,:\, u(x)>0\}\ne \emptyset$.
Then $\Delta u$ is a non-negative measure on $U$, so
Proposition~\ref{lem:trivialmp} is applicable to each cell contained
in $U$.  Moreover the proof of Proposition~\ref{lem:trivialmp}  
implies there cannot be a strict maximum at $p\in V_*$ if there is a
scale $n$ such that all neighbors $q\sim_n p$ are in $U$. Thus the
maximum of $u$ cannot only occur at an interior point to $U$ because
every such point has a neighborhood in $U$ that is a cell or finite
union of cells at a single scale.  Since $u=0$ at any point
$p\in\partial U$ that is interior to $E$, the maximum must be achieved
on $\partial U\cap \partial E$, which implies the stated inequality. 

If a positive maximum occurs at an interior point  $x$ of $E$ it must
also be that $x\in U$, which is open.  If $x\not\in V_*$ it is
contained in the interior of a cell contained in $U$ and from
Proposition~\ref{lem:trivialmp} $u$ is constant on this cell.  If
$x\in V_*$ then there is neighborhood of $x$ in $U$ consisting of $x$
and the interiors of some cells meeting at $x$ and lying in $U$, but
in this case the reasoning in the proof of
Proposition~\ref{lem:trivialmp} implies $u$ is constant on these
cells.  In summary, $u=u(x)$ on a neighborhood of $x$.  However this
implies the set $Y=\{y:u(y)=u(x)\}$ is open in $E$ because every such
$y$ is necessarily in the open set $U$, and $Y$ is obviously closed
because $u$ is continuous, so $u$ is constant on the connected
component of $E$ containing $x$. 
\end{proof}


In the classical setting of a Euclidean space, one standard approach
to obtaining the strong maximum principle from the weak maximum
principle is to use the Hopf lemma.  It is perhaps interesting to note
that in the fractal setting we can prove a Hopf-type lemma at points
in $V_*$ but have no corresponding results at points of $K\setminus
V_*$ and therefore cannot use this approach to obtain a strong maximum
principle. 

To state our Hopf-type lemma we recall that the normal
derivative~\cite[Definition 3.7.6]{Kig_CUP01} of a function at a point
$p\in V_0$ may be written using the scale~$n$ piecewise harmonic
function $h_p^{(n)}$ which is $1$ at $p$ and zero on $V_n\setminus\{p\}$
as  
\begin{equation}\label{eqn:defd_n}
	\partial_n u(p)= \lim_{n\to\infty} \E(h_p^{(n)},u)
\end{equation}
There is no analogue of the normal derivative at points in $K\setminus
V_*$, thus no Hopf lemma at these points. 

\begin{lem}
Let $\nu$ be a non-negative Radon measure on $K$ and suppose
$u\in\dom_{\mathcal{M}} \Delta$ satisfies $Lu=\Delta u -u\nu$ is also
non-negative measure. If there is $p\in V_0$ such that $u(p)>u(x)$ for
all $x\in K\setminus \{p\}$ and also $u(p)>0$ then $\partial_nu(p)>0$. 
\end{lem}
\begin{proof}
If $n<m$ the difference $k^{(n,m)}=h_p^{(n)}-h_p^{(m)}$ is zero at $p$
and the points $q\sim_n p$, and is equal to $h_p^{(n)}>0$ at each
$q\sim_m p$; it is otherwise harmonic, so the minimum principle for
harmonic functions implies it is non-negative.  Now $u(p)>0$ and $u$
is continuous so there is $n$ such that $u$ is positive on the support
of $k^{(n,m)}$.  For this $n$ and any $m>n$ we see $\Delta u\geq u\nu$
is a non-negative measure on the support of $k^{(n,m)}$ and therefore 
\begin{equation*}
	\E(h_p^{(n)}-h_p^{(m)},u) = - \int \bigl(h_p^{(n)}-h_p^{(m)}\bigr)\Delta u \leq 0
\end{equation*}
which gives $\partial_nu(p)=\lim_{m\to\infty}\E(h_p^{(m)},u)\geq
\E(h_p^{(n)},u)=\sum_{q\sim_n p} c_{pq} (u(p)-u(q))$ for some values
$c_{pq}>0$ depending on $\E_n$.  The fact that $u(p)>u(q)$ for all $q$
concludes the proof. 
\end{proof}

\section{Harnack Inequality}\label{sec:harnack}

In the classical setting the strong maximum principle for $Lu\geq0$ implies a Harnack inequality for solutions of $Lu=0$, see~\cite[Section~8.8]{GTbook}. We show that this is the case in our
setting.

Before stating the Harnack inequality we note that the equation $Lu=0$ for $L$ as in~\eqref{eqn:defL} has solutions for sufficiently small measures $\nu$ on $K$ because with boundary data a harmonic function $h$ the operator defined using the continuous non-negative Green kernel $g(x,y)$ by
\begin{equation*}
 u\mapsto h + \int g(x,y) u(y) \nu(dy)
\end{equation*}
is contracting in the uniform norm provided $\int g(x,y)\nu(dy)<1$.  Moreover, on the cell $F_\omega(K)$ the Green kernel is $r_\omega g(F_\omega^{-1}(x),F_\omega^{-1}(y))$ where $r_\omega=\prod_{j=1}^{|\omega|} r_{\omega_j}$ by~\cite[Proposition~3.5.5]{Kig_CUP01}.  It follows for any $\nu$ that we can take $|\omega|$ sufficiently large  so that contractivity of the analogous operator on $F_\omega(K)$ is valid, ensuring local solutions exist everywhere.

\begin{thm}[Harnack Inequality]\label{thm:harnack}
On $K$, fix a non-negative finite Radon measure $\nu$.  Suppose
$u\in\dom_{\mathcal{M}}(\Delta)$ satisfies both $u\geq0$ and  $Lu=0$,
where $L$ is as in~\eqref{eqn:defL} and the latter is an equality of
measures.  If $E$ is a compact subset of $K\setminus V_0$ then there
is $C$ that depends only on $E$ such that $\max_E u\leq C \min_E u$. 
\end{thm}

The proof makes use of the following result, which may be of independent interest. 
\begin{prop}\label{prop:unifcpt}
For $\nu$ a non-negative finite Radon measure  on $K$ let
\begin{equation*}
\mathcal{A}=\{u\in\dom_{\mathcal{M}}(\Delta): Lu=0 \text{ as measures and } 0\leq u\leq1\}.
\end{equation*}
Then $\mathcal{A}$ is compact in the uniform norm.
\end{prop}
\begin{proof}
For any $u\in\mathcal{A}$  with $u|_{V_0}=0$ we have
\begin{equation*}
	\DF(u,u)=-\int u\Delta u=-\int u\,d\sigma=-\int u^2d\nu
	\end{equation*}
so that $|u|\leq1$ implies $|\E(u,u)|\leq \nu(K)=M<\infty$.
Applying the estimate \eqref{eqn:estres} we have
\begin{equation}\label{eqn:easybound}
	|u(x)-u(y)|^2\leq\DF(u,u)R(x,y)\leq  M  R(x,y).
\end{equation}

If instead, $u|_{V_0}\neq0$, then choose a harmonic function $h$ so $(u-h)|_{V_0}=0$ and since $|h|\leq1$ on $V_0$ we have $|h(x)-h(y)|^2\leq 2R(x,y)$, so~\eqref{eqn:easybound} holds with $M$ replaced by $M+2$.  However an estimate of the type~\eqref{eqn:easybound} implies $\mathcal{A}$ is equicontinuous, and since the definition of $\mathcal{A}$ implies it is equibounded, an application of the Arzela-Ascoli theorem then yields that it is precompact in the uniform norm.

Now suppose $\{u_n\}\subset\mathcal{A}$ and $u_n\to u$ uniformly.  If $v\in \dom(\DF)$  and $v|{V_0}=0$ we know $v$ is continuous and therefore bounded so Jensen's inequality provides
\begin{equation*}
	\bigl|\DF(u_n-u_m,v)\bigr| = \Bigl|\int v\Delta(u_n-u_m)\Bigr|
	\leq M \|v\|_\infty \|u_n-u_m\|_\infty ,
	\end{equation*}
and therefore $u_n$ is Cauchy in $\dom(\DF)$.  However $\dom(\DF)$ is a Hilbert space under the norm $\DF+\|\cdot\|_\infty$, so we conclude that $u\in\dom(\DF)$ and $\DF(u,v)=-\int v u d\nu$, whence $\Delta u=u\nu$.  This implies $u\in\mathcal{A}$, so $\mathcal{A}$ is closed, and in light of our Arzela-Ascoli argument, compact in the uniform norm.
\end{proof}

\begin{proof}[Proof of Theorem~\protect{\ref{thm:harnack}}]
Let
\begin{equation*}
	\mathcal{A}=\{u\in\dom(\Delta): Lu=0,\ 0\leq u\leq 1\text{ and } \max_{V_0}u=1\}.
	\end{equation*}
Evidently this is a closed subset of the space of functions considered in Proposition~\ref{prop:unifcpt} and is therefore compact in the uniform norm.  Hence the quantity $c=\inf_\mathcal{A} \min_{x\in E}u(x)$ is achieved by some $\tilde{u}\in\mathcal{A}$.  Either $c=1$ or $\tilde{u}$ is non-constant and the hypothesis $\tilde{u}\geq0$ implies $c>0$ by the strong minimum principle.

Now take $u$ as in the statement of the theorem.  It is continuous and the strong maximum principle is applicable, so it achieves its maximum on $V_0$.  Accordingly, $u/(\max_{V_0} u)\in\mathcal{A}$ and has minimum at least $c$ on $E$.  The result follows with $C=1/c$.
\end{proof}

\begin{rem}
The above immediately implies the corresponding result in the setting of a bounded subset of $K_\infty$ because such a set is contained in a (sufficiently large) copy of $K$ and we can transfer the theorem directly to this setting using the self-similarity of the energy.
\end{rem}

\section*{Acknowledgments.}   K. A. O.\ was partially supported by a
grant from the Simons Foundation $\# 319197$,  the U. S.\ Army
Research Office  grant  W911NF1610008,  the National Science
Foundation grant DMS 1814253, and an MLK  visiting
professorship. L. G. R. was partially supported by the National
Science Foundation grant DMS 1659643.

\end{document}